\documentclass[12pt,reqno]{article}

\usepackage[usenames]{color}
\usepackage{amssymb}
\usepackage{graphicx}
\usepackage{amscd}

\usepackage{amsthm}
\newtheorem{theorem}{Theorem}
\newtheorem{lemma}[theorem]{Lemma}
\newtheorem{proposition}[theorem]{Proposition}
\newtheorem{corollary}[theorem]{Corollary}

\theoremstyle{definition}

\newtheorem{example}[theorem]{Example}

\usepackage[colorlinks=true,
linkcolor=webgreen, filecolor=webbrown,
citecolor=webgreen]{hyperref}

\definecolor{webgreen}{rgb}{0,.5,0}
\definecolor{webbrown}{rgb}{.6,0,0}

\usepackage{color}

\usepackage{float}

\usepackage{graphics,amsmath,amssymb}
\usepackage{amsfonts}
\usepackage{latexsym}
\usepackage{epsf}

\newcommand{\seqnum}[1]{\href{http://www.research.att.com/cgi-bin/access.cgi/as/~njas/sequences/eisA.cgi?Anum=#1}{\underline{#1}}}

\begin{document}

\begin{center}
\vskip 1cm{\LARGE\bf Riordan arrays, orthogonal polynomials as moments, and Hankel transforms} \vskip 1cm \large
Paul Barry\\
School of Science\\
Waterford Institute of Technology\\
Ireland\\
\href{mailto:pbarry@wit.ie}{\tt pbarry@wit.ie} \\

\end{center}
\vskip .2 in

\begin{abstract} Taking the examples of Legendre and Hermite orthogonal polynomials, we show how to interpret the fact that these
orthogonal polynomials are moments of other orthogonal polynomials in terms of their associated Riordan arrays. We use these means to
calculate the Hankel transforms of the associated polynomial sequences.
\end{abstract}
\section{Introduction}

In this note, we shall re-interpret some of the results  of Ismail and Stanton \cite{Ismail_Stanton_1, Ismail_Stanton_2} in terms of
Riordan arrays. These  authors  give functionals \cite{Ismail_Stanton_1} whose moments are the Hermite, Laguerre, and various Meixner
families of polynomials. In this note, we shall confine ourselves to Legendre and Hermite polynomials. Indeed, the types of orthogonal
polynomials representable with Riordan arrays is very limited (see below), but it is nevertheless instructive to show that a number of
them can be exhibited as moments, again using (parameterized) Riordan arrays.

The essence of the paper is to show that a Riordan array $L$ (either ordinary or exponential) defines a family of orthogonal polynomials (via its inverse $L^{-1}$) if and only if its production matrix \cite{DeutschShap, ProdMat_0, ProdMat} is tri-diagonal. The sequence of moments $\mu_n$ associated to the family of orthogonal polynomials then appears as the elements of the first column of $L$. In terms of generating functions, this means that if $L=(g,f)$ (or $L=[g,f]$), then $g(x)$ is the generating function of the moment sequence. By defining suitable parameterized Riordan arrays, we can exhibit the Legendre and Hermite polynomials as such moment sequences.

While partly expository in nature, the note assumes a certain familiarity with integer sequences, generating functions, orthogonal
polynomials \cite{Chihara, wgautschi, Szego}, Riordan arrays \cite{SGWW, Spru}, production matrices \cite{ProdMat, P_W}, and the
Hankel transform of sequences \cite{BRP, CRI, Layman}. We provide background material in this note to give a hopefully coherent
narrative.
Many interesting examples of sequences and Riordan arrays can be found in Neil Sloane's On-Line
Encyclopedia of Integer Sequences (OEIS), \cite{SL1, SL2}. Sequences are frequently referred to by their
OEIS number. For instance, the binomial matrix $\mathbf{B}$ (``Pascal's triangle'') is \seqnum{A007318}.
\newline\newline
\noindent The plan of the paper is as follows:
\begin{enumerate}
\item This Introduction
\item Preliminaries on integer sequences and (ordinary) Riordan arrays
\item Orthogonal polynomials and Riordan arrays
\item Exponential Riordan arrays and orthogonal polynomials
\item The Hankel transform of an integer sequence
\item Legendre polynomials
\item Legendre polynomials as moments
\item Hermite polynomials
\item Hermite polynomials as moments
\item Acknowledgements
\item Appendix - The Stieltjes transform of a measure

\end{enumerate}

\section{Preliminaries on integer sequences and Riordan arrays}
For an integer sequence $a_n$, that is, an element of $\mathbb{Z}^\mathbb{N}$, the power series
$f(x)=\sum_{n=0}^{\infty}a_n x^n$ is called the \emph{ordinary generating function} or g.f. of the sequence.
$a_n$ is thus the coefficient of $x^n$ in this series. We denote this by
$a_n=[x^n]f(x)$. For instance, $F_n=[x^n]\frac{x}{1-x-x^2}$ is the $n$-th Fibonacci number \seqnum{A000045}, while
$C_n=[x^n]\frac{1-\sqrt{1-4x}}{2x}$ is the $n$-th Catalan number \seqnum{A000108}. The article \cite{Merlini_MC} gives examples of the use of the operator $[x^n]$. We use the notation
$0^n=[x^n]1$ for the sequence $1,0,0,0,\ldots,$ \seqnum{A000007}. Thus $0^n=[n=0]=\delta_{n,0}=\binom{0}{n}$. Here,
we have used the Iverson bracket notation \cite{Concrete},
defined by $[\mathcal{P}]=1$ if the proposition $\mathcal{P}$
is true, and
$[\mathcal{P}]=0$ if $\mathcal{P}$ is false.

For a power series
$f(x)=\sum_{n=0}^{\infty}a_n x^n$ with $f(0)=0$ we define the reversion or compositional inverse of $f$ to be the
power series $\bar{f}(x)$ such that $f(\bar{f}(x))=x$. We shall sometimes write this as
$\bar{f}= \text{Rev}f$.
\newline\newline
For a lower triangular matrix $(a_{n,k})_{n,k \ge 0}$ the row sums give the sequence with general term
$\sum_{k=0}^n a_{n,k}$ while the diagonal sums form the sequence with general term
$$\sum_{k=0}^{\lfloor \frac{n}{2} \rfloor} a_{n-k,k}.$$
\noindent
The \emph{Riordan group} \cite{SGWW, Spru}, is a set of
infinite lower-triangular integer matrices, where each matrix is
defined by a pair of generating functions
$g(x)=1+g_1x+g_2x^2+\cdots$ and $f(x)=f_1x+f_2x^2+\cdots$ where
$f_1\ne 0$ \cite{Spru}. We assume in addition that $f_1=1$ in what follows. The associated matrix is the matrix whose
$i$-th column is generated by $g(x)f(x)^i$ (the first column being
indexed by 0). The matrix corresponding to the pair $g, f$ is
denoted by $(g, f)$ or $\cal{R}$$(g,f)$. The group law is then given
by
\begin{displaymath} (g, f)\cdot(h, l)=(g, f)(h, l)=(g(h\circ f), l\circ
f).\end{displaymath} The identity for this law is $I=(1,x)$ and the
inverse of $(g, f)$ is $(g, f)^{-1}=(1/(g\circ \bar{f}), \bar{f})$
where $\bar{f}$ is the compositional inverse of $f$.

A Riordan array of the form $(g(x),x)$, where $g(x)$ is the
generating function of the sequence $a_n$, is called the
\emph{sequence array} of the sequence $a_n$. Its general term is
$a_{n-k}$ (or more precisely, $[k \le n] a_{n-k}$). Such arrays are also called \emph{Appell} arrays as they form the elements of the so-called
Appell subgroup.
\newline\newline If $\mathbf{M}$ is the matrix $(g,f)$, and
$\mathbf{a}=(a_0,a_1,\ldots)'$ is an integer sequence with ordinary
generating function $\cal{A}$ $(x)$, then the sequence
$\mathbf{M}\mathbf{a}$ has ordinary generating function
$g(x)$$\cal{A}$$(f(x))$. The (infinite) matrix $(g,f)$ can thus be considered to act on the ring of
integer sequences $\mathbb{Z}^\mathbb{N}$ by multiplication, where a sequence is regarded as a
(infinite) column vector. We can extend this action to the ring of power series
$\mathbb{Z}[[x]]$ by
$$(g,f):\cal{A}(\mathnormal{x}) \mapsto \mathnormal{(g,f)}\cdot
\cal{A}\mathnormal{(x)=g(x)}\cal{A}\mathnormal{(f(x))}.$$
\begin{example} The so-called \emph{binomial matrix} $\mathbf{B}$ is the element
$(\frac{1}{1-x},\frac{x}{1-x})$ of the Riordan group. It has general
element $\binom{n}{k}$, and hence as an array coincides with Pascal's triangle. More generally, $\mathbf{B}^m$ is the
element $(\frac{1}{1-m x},\frac{x}{1-mx})$ of the Riordan group,
with general term $\binom{n}{k}m^{n-k}$. It is easy to show that the
inverse $\mathbf{B}^{-m}$ of $\mathbf{B}^m$ is given by
$(\frac{1}{1+mx},\frac{x}{1+mx})$.
\end{example}
\begin{example} If $a_n$ has generating function $g(x)$, then the generating function of the
sequence $$b_n=\sum_{k=0}^{\lfloor \frac{n}{2} \rfloor} a_{n-2k}$$ is equal to
$$\frac{g(x)}{1-x^2}=\left(\frac{1}{1-x^2},x\right)\cdot g(x),$$ while the generating function of the sequence
$$d_n=\sum_{k=0}^{\lfloor \frac{n}{2} \rfloor} \binom{n-k}{k} a_{n-2k}$$ is equal to
$$\frac{1}{1-x^2}g\left(\frac{x}{1-x^2}\right)=\left(\frac{1}{1-x^2},\frac{x}{1-x^2}\right)\cdot g(x).$$
\end{example}
\noindent The row sums of the matrix $(g, f)$ have generating function
$$(g,f)\cdot \frac{1}{1-x}=\frac{g(x)}{1-f(x)}$$
 while the diagonal sums of $(g, f)$ (sums of left-to-right diagonals in the North East direction) have generating
function $g(x)/(1-xf(x))$. These coincide with the row sums of the ``generalized'' Riordan array $(g,xf)$:
$$(g,xf)\cdot\frac{1}{1-x}=\frac{g(x)}{1-xf(x)}.$$ \noindent For instance the
Fibonacci numbers $F_{n+1}$ are the diagonal sums of the binomial matrix $\mathbf{B}$ given by
$\left(\frac{1}{1-x},\frac{x}{1-x}\right)$\,:
\begin{displaymath}\left(\begin{array}{ccccccc} 1 & 0 &
0
& 0 & 0 & 0 & \ldots \\1 & 1 & 0 & 0 & 0 & 0 & \ldots \\ 1 & 2
& 1 & 0 & 0 &
0 & \ldots \\ 1 & 3 & 3 & 1 & 0 & 0 & \ldots \\ 1 & 4 & 6
& 4 & 1 & 0 & \ldots \\1 & 5 & 10 & 10 & 5 & 1
&\ldots\\
\vdots &
\vdots & \vdots & \vdots & \vdots & \vdots &
\ddots\end{array}\right)\end{displaymath} while they are the row sums of the
``generalized'' or ``stretched'' \cite{CMS} Riordan array
$\left(\frac{1}{1-x},\frac{x^2}{1-x}\right)$\,:
\begin{displaymath}\left(\begin{array}{ccccccc} 1 & 0 &
0
& 0 & 0 & 0 & \ldots \\1 & 0 & 0 & 0 & 0 & 0 & \ldots \\ 1 & 1
& 0 & 0 & 0 &
0 & \ldots \\ 1 & 2 & 0 & 0 & 0 & 0 & \ldots \\ 1 & 3 & 1
& 0 & 0 & 0 & \ldots \\1 & 4 & 3 & 0 & 0 & 0
&\ldots\\
\vdots &
\vdots & \vdots & \vdots & \vdots & \vdots &
\ddots\end{array}\right).\end{displaymath}
\noindent

\noindent Each Riordan array $(g(x),f(x))$ has bi-variate generating function given by
$$\frac{g(x)}{1-yf(x)}.$$
For instance, the binomial matrix $\mathbf{B}$ has generating function
$$\frac{\frac{1}{1-x}}{1-y\frac{x}{1-x}}=\frac{1}{1-x(1+y)}.$$
\newline\newline
\noindent For a sequence $a_0, a_1, a_2, \ldots$ with g.f. $g(x)$, the ``aeration'' of the sequence is the sequence
$a_0, 0, a_1, 0, a_2, \ldots$ with interpolated zeros. Its g.f. is $g(x^2)$.

The aeration of a (lower-triangular) matrix $\mathbf{M}$ with general term $m_{i,j}$ is the matrix whose general term is given by
$$m^r_{\frac{i+j}{2},\frac{i-j}{2}}\frac{1+(-1)^{i-j}}{2},$$ where
$m^r_{i,j}$ is the $i,j$-th element of the reversal of $\mathbf{M}$:
$$m^r_{i,j}=m_{i,i-j}.$$
In the case of a Riordan array (or indeed any lower triangular array), the row sums of the aeration are equal to the diagonal sums of
the reversal of the original matrix.
\begin{example}
The Riordan array $(c(x^2), xc(x^2))$ is the aeration of $(c(x),xc(x))$ \seqnum{A033184}. Here
$$c(x)=\frac{1-\sqrt{1-4x}}{2x}$$ is the g.f. of the Catalan numbers.
Indeed, the reversal of $(c(x), xc(x))$ is the matrix with general element
$$[k\le n+1] \binom{n+k}{k}\frac{n-k+1}{n+1},$$ which begins

\begin{displaymath}\left(\begin{array}{ccccccc} 1 & 0 &
0
& 0 & 0 & 0 & \ldots \\1 & 1 & 0 & 0 & 0 & 0 & \ldots \\ 1 & 2
& 2 & 0 & 0 &
0 & \ldots \\ 1 & 3 & 5 & 5 & 0 & 0 & \ldots \\ 1 & 4 & 9
& 14 & 14 & 0 & \ldots \\1 & 5 & 14 & 28 & 42 & 42
&\ldots\\
\vdots &
\vdots & \vdots & \vdots & \vdots & \vdots &
\ddots\end{array}\right).\end{displaymath} This is \seqnum{A009766}.
Then $(c(x^2),xc(x^2))$ has general element
$$\binom{n+1}{\frac{n-k}{2}}\frac{k+1}{n+1}\frac{1+(-1)^{n-k}}{2},$$ and begins
\begin{displaymath}\left(\begin{array}{ccccccc} 1 & 0 &
0
& 0 & 0 & 0 & \ldots \\0 & 1 & 0 & 0 & 0 & 0 & \ldots \\ 1 & 0
& 1 & 0 & 0 &
0 & \ldots \\ 0 & 2 & 0 & 1 & 0 & 0 & \ldots \\ 2 & 0 & 3
& 0 & 1 & 0 & \ldots \\0 & 5 & 0 & 4 & 0 & 1
&\ldots\\
\vdots &
\vdots & \vdots & \vdots & \vdots & \vdots &
\ddots\end{array}\right).\end{displaymath} This is \seqnum{A053121}.
Note that
$$(c(x^2),xc(x^2))=\left(\frac{1}{1+x^2},\frac{x}{1+x^2}\right)^{-1}.$$
\noindent We observe that the diagonal sums of the reverse of $(c(x),xc(x))$ coincide with the row sums of
$(c(x^2),xc(x^2))$, and are equal to the central binomial coefficients $\binom{n}{\lfloor \frac{n}{2} \rfloor}$ \seqnum{A001405}.
\end{example}
An important feature of Riordan arrays is that they have a number of sequence characterizations \cite{Cheon, He}. The simplest of
these
is as follows.
\begin{proposition} \label{Char} \cite[Theorem 2.1, Theorem 2.2]{He} Let $D=[d_{n,k}]$ be an infinite triangular matrix. Then $D$ is a Riordan array if and only if there
exist two sequences $A=[a_0,a_1,a_2,\ldots]$ and $Z=[z_0,z_1,z_2,\ldots]$ with $a_0 \neq 0$, $z_0 \neq 0$ such that
\begin{itemize}
\item $d_{n+1,k+1}=\sum_{j=0}^{\infty} a_j d_{n,k+j}, \quad (k,n=0,1,\ldots)$
\item $d_{n+1,0}=\sum_{j=0}^{\infty} z_j d_{n,j}, \quad (n=0,1,\ldots)$.
\end{itemize}
\end{proposition}
The coefficients $a_0,a_1,a_2,\ldots$ and $z_0,z_1,z_2,\ldots$ are called the $A$-sequence and the $Z$-sequence of the Riordan array
$D=(g(x),f(x))$, respectively.
Letting $A(x)$ be the generating function of the $A$-sequence and $Z(x)$ be the generating function of the $Z$-sequence, we have
\begin{equation}\label{AZ_eq} A(x)=\frac{x}{\bar{f}(x)}, \quad Z(x)=\frac{1}{\bar{f}(x)}\left(1-\frac{1}{g(\bar{f}(x))}\right).\end{equation}
\section{Orthogonal polynomials and Riordan arrays}
By an
\emph{orthogonal polynomial sequence} $(p_n(x))_{n \ge 0}$ we shall
understand \cite{Chihara, wgautschi} an infinite sequence of
polynomials
$p_n(x)$, $n\ge 0$, of degree $n$, with real coefficients (often integer
coefficients) that are mutually orthogonal on an interval
$[x_0,x_1]$ (where
$x_0=-\infty$ is allowed, as well as $x_1=\infty$), with
respect to a weight function $w:[x_0,x_1] \to \mathbb{R}$ \,:
$$\int_{x_0}^{x_1}
p_n(x)p_m(x)w(x)dx = \delta_{nm}\sqrt{h_nh_m},$$ where
$$\int_{x_0}^{x_1} p_n^2(x)w(x)dx=h_n.$$ We assume that $w$ is strictly positive on the interval $(x_0,x_1)$.  Every such
sequence
obeys a so-called
``three-term recurrence" \,: $$
p_{n+1}(x)=(a_nx+b_n)p_n(x)-c_np_{n-1}(x)$$ for coefficients
$a_n$, $b_n$ and $c_n$ that depend on $n$ but
not $x$. We note that if $$p_j(x)=k_jx^j+k'_jx^{j-1}+\ldots
\qquad j=0,1,\ldots$$ then $$a_n=\frac{k_{n+1}}{k_n},\qquad
b_n=a_n\left(\frac{k'_{n+1}}{k_{n+1}}-\frac{k'_n}{k_n}\right),
\qquad c_n=a_n\left(\frac{k_{n-1}h_n}{k_nh_{n-1}}\right).$$

Since the degree
of $p_n(x)$ is $n$, the coefficient array of the polynomials
is
a lower triangular (infinite) matrix. In the case of monic
orthogonal
polynomials (where $k_n=1$ for all $n$) the diagonal elements of this array will all be
$1$. In this case, we can write the three-term recurrence as
$$p_{n+1}(x)=(x-\alpha_n)p_n(x)-\beta_n p_{n-1}(x), \qquad
p_0(x)=1,\qquad p_1(x)=x-\alpha_0.$$

\noindent The \emph{moments} associated to the orthogonal
polynomial sequence are the numbers $$\mu_n=\int_{x_0}^{x_1}
x^n w(x)dx.$$

\noindent We can find $p_n(x)$, $\alpha_n$ and $\beta_n$ (and in the right circumstances, $w(x)$ - see the Appendix) from a
knowledge
of these moments. To do this, we let $\Delta_n$ be the Hankel
determinant
$|\mu_{i+j}|_{i,j\ge 0}^n$ and $\Delta_{n,x}$ be the same
determinant, but with the last row equal to $1,x,x^2,\ldots$.
Then
$$p_n(x)=\frac{\Delta_{n,x}}{\Delta_{n-1}}.$$ More generally,
we let $H\left(\begin{array}{ccc} u_1&\ldots& u_k\\
v_1 & \ldots & v_k\end{array}\right)$ be the determinant
of
Hankel type with $(i,j)$-th term $\mu_{u_i+v_j}$. Let
$$\Delta_n=H\left(\begin{array}{cccc} 0&1&\ldots&n\\
0&1&\ldots&n\end{array}\right),\qquad
\Delta'_n=H_n\left(\begin{array}{ccccc}
0&1&\ldots&n-1&n\\ 0&1&\ldots&n-1&n+1\end{array}\right).$$
Then
we have
$$\alpha_n=\frac{\Delta'_n}{\Delta_n}-\frac{\Delta'_{n-1}}{\Delta_{n-1}},\qquad
\beta_n=\frac{\Delta_{n-2} \Delta_n}{\Delta_{n-1}^2}.$$

We shall say that a family of polynomials $\{p_n(x)\}_{n \ge 0}$ is \cite{Viennot} \emph{formally orthogonal}, if there exists a linear functional $\mathcal{L}$ on polynomials such that
\begin{enumerate}
\item $p_n(x)$ is a polynomial of degree $n$,
\item $\mathcal{L}(p_n(x)p_m(x))=0$ for $m\ne n$,
\item $\mathcal{L}(p_n^2(x)) \ne 0$.
\end{enumerate}
Consequences of this definition include \cite{Viennot} that
$$\mathcal{L}(x^m p_n(x))=\kappa_n \delta_{mn}, \quad 0 \le m \le n,\quad \kappa_n \ne 0,$$ and if
$q(x)=\sum_{k=0}^n a_k p_k(x)$, then $a_k=\mathcal{L}(qp_k)/\mathcal{L}(p_k^2)$.

The sequence of numbers $\mu_n=\mathcal{L}(x^n)$ is called the sequence of moments of the family of orthogonal polynomials defined by
$\mathcal{L}$.
Note that where a suitable weight function $w(x)$ exists, then we can realize the functional $\mathcal{L}$ as
$$\mathcal{L}(p(x))=\int_{\mathbb{R}} p(x)w(x)\,dx.$$
\noindent If the family $p_n(x)$ is an orthogonal family for the functional $\mathcal{L}$, then it is also orthogonal for $c\mathcal{L}$, where $c \ne 0$. In the sequel, we shall always assume that $\mu_0=\mathcal{L}(p_0(x))=1$.

The following well-known results (the first is the well-known ``Favard's Theorem''), which we essentially reproduce from
\cite{Kratt}, specify the links between orthogonal polynomials, three term recurrences, and the recurrence coefficients and the g.f. of
the moment sequence of the orthogonal polynomials.
\begin{theorem} \label{ThreeT}\cite{Kratt} (Cf. \cite{Viennot}, Th\'eor\`eme $9$ on p.I-4, or \cite{Wall}, Theorem $50.1$). Let $(p_n(x))_{n\ge 0}$
be a sequence of monic polynomials, the polynomial $p_n(x)$ having degree $n=0,1,\ldots$ Then the sequence $(p_n(x))$ is (formally)
orthogonal if and only if there exist sequences $(\alpha_n)_{n\ge 0}$ and $(\beta_n)_{n\ge 1}$ with $\beta_n \neq 0$ for all $n\ge 1$,
such that the three-term recurrence
$$p_{n+1}=(x-\alpha_n)p_n(x)-\beta_n p_{n-1}(x), \quad \text{for}\quad n\ge 1, $$
holds, with initial conditions $p_0(x)=1$ and $p_1(x)=x-\alpha_0$.
\end{theorem}

\begin{theorem} \label{CF} \cite{Kratt} (Cf. \cite{Viennot}, Proposition 1, (7), on p. V-$5$, or \cite{Wall}, Theorem $51.1$). Let $(p_n(x))_{n\ge
0}$ be a sequence of monic polynomials, which is orthogonal with respect to some functional $\mathcal{L}$. Let
$$p_{n+1}=(x-\alpha_n)p_n(x)-\beta_n p_{n-1}(x), \quad \text{for}\quad n\ge 1, $$ be the corresponding three-term recurrence which is
guaranteed by Favard's theorem. Then the generating function
$$g(x)=\sum_{k=0}^{\infty} \mu_k x^k $$ for the moments $\mu_k=\mathcal{L}(x^k)$ satisfies
$$g(x)=\cfrac{\mu_0}{1-\alpha_0 x-
\cfrac{\beta_1 x^2}{1-\alpha_1 x -
\cfrac{\beta_2 x^2}{1-\alpha_2 x -
\cfrac{\beta_3 x^2}{1-\alpha_3 x -\cdots}}}}.$$
\end{theorem}

Given a family of monic orthogonal polynomials
$$p_{n+1}(x)=(x-\alpha_n)p_n(x)-\beta_n p_{n-1}(x), \qquad
p_0(x)=1,\qquad p_1(x)=x-\alpha_0,$$
we can write $$p_n(x)=\sum_{k=0}^n a_{n,k}x^k.$$ Then we have
$$\sum_{k=0}^{n+1}a_{n+1,k}x^k=(x-\alpha_n)\sum_{k=0}^n
a_{n,k}x^k - \beta_n
\sum_{k=0}^{n-1}a_{n-1,k}x^k$$ from which we deduce
\begin{equation}\label{OP_1}a_{n+1,0}=-\alpha_n
a_{n,0}-\beta_n
a_{n-1,0}\end{equation}
and \begin{equation}\label{OP_2}a_{n+1,k}=a_{n,k-1}-\alpha_n
a_{n,k}-\beta_n a_{n-1,k}\end{equation}
We note that if $\alpha_n$ and
 $\beta_n$ are constant, equal to $\alpha$ and $\beta$, respectively, then the sequence $(1,-\alpha,-\beta,0,0,\ldots)$ forms an $A$-sequence for the coefficient array.
The question immediately arises as to the conditions under which a Riordan array $(g,f)$ can be the coefficient array
 of a family of orthogonal polynomials. A partial answer is given by the following proposition.

\begin{proposition} Every Riordan array of the form
$$\left(\frac{1}{1+rx+sx^2},\frac{x}{1+rx+sx^2}\right)$$ is the coefficient array of a family of monic orthogonal polynomials.
\end{proposition}
\begin{proof} The array $\left(\frac{1}{1+rx+sx^2},\frac{x}{1+rx+sx^2}\right)$ \cite{Jin} has a $C$-sequence $C(x)=\sum_{n\ge
0}c_n
x^n$ given by
$$\frac{x}{1+rx+sx^2}=\frac{x}{1-xC(x)},$$ and thus $$C(x)=-r-sx.$$ This means that the Riordan array
$\left(\frac{1}{1+rx+sx^2},\frac{x}{1+rx+sx^2}\right)$ is determined by the fact that
$$a_{n+1,k}=a_{n,k-1}+\sum_{i \ge 0} c_i a_{n-i,k} \quad \text{for $n,k=0,1,2,\ldots$}$$ where
$a_{n,-1}=0$.
In the case of $\left(\frac{1}{1+rx+sx^2},\frac{x}{1+rx+sx^2}\right)$ we have $$a_{n+1,k}=a_{n,k-1}-ra_{n,k}-sa_{n-1,k}.$$
Working backwards, this now ensures that
$$p_{n+1}(x)=(x-r)p_n(x)-s p_{n-1}(x),$$ where
$p_n(x)=\sum_{k=0}^n a_{n,k}x^n$. The result now follows from Theorem \ref{ThreeT}.
\end{proof}
We note that in this case the three-term recurrence coefficients $\alpha_n$ and $\beta_n$ are constants.
We have in fact the following proposition (see the next section for information on the Chebyshev polynomials).
\begin{proposition} The Riordan array $\left(\frac{1}{1+rx+sx^2},\frac{x}{1+rx+sx^2}\right)$ is the coefficient array of the modified
Chebyshev polynomials of the second kind given by
$$P_n(x)=(\sqrt{s})^n U_n\left(\frac{x-r}{2\sqrt{s}}\right), \quad n=0,1,2,\ldots$$
\end{proposition}
\begin{proof}
The production array of $\left(\frac{1}{1+rx+sx^2},\frac{x}{1+rx+sx^2}\right)^{-1}$ is given by
\begin{displaymath}\left(\begin{array}{ccccccc} r & 1 &
0
& 0 & 0 & 0 & \ldots \\s & r & 1 & 0 & 0 & 0 & \ldots \\ 0 & s
& r & 1 & 0 &
0 & \ldots \\ 0 & 0 & s & r & 1 & 0 & \ldots \\ 0 & 0 & 0
& s & r & 1 & \ldots \\0 & 0 & 0 & 0 & s & r
&\ldots\\
\vdots &
\vdots & \vdots & \vdots & \vdots & \vdots &
\ddots\end{array}\right).\end{displaymath}
The result is now a consequence of the article \cite{Elouafi} by Elouafi, for instance.
\end{proof}

The complete answer can be found by considering the associated production matrix of a Riordan arrray, in the following sense.

The
concept of a \emph{production matrix} \cite{DeutschShap, ProdMat_0,
ProdMat}
is a general one, but for this work we find it convenient to
review it in
the context of Riordan arrays. Thus let $P$ be an infinite
matrix (most often it will have integer entries). Letting
$\mathbf{r}_0$
be the row vector
$$\mathbf{r}_0=(1,0,0,0,\ldots),$$ we define $\mathbf{r}_i=\mathbf{r}_{i-1}P$, $i \ge 1$.
Stacking these rows leads to another infinite matrix which we
denote by
$A_P$. Then $P$ is said to be the \emph{production matrix} for
$A_P$.

\noindent If we let $$u^T=(1,0,0,0,\ldots,0,\ldots)$$ then we
have $$A_P=\left(\begin{array}{c}
u^T\\u^TP\\u^TP^2\\\vdots\end{array}\right)$$ and
$$\bar{I}A_P=A_PP$$ where $\bar{I}=(\delta_{i+1,j})_{i,j \ge 0}$ (where
$\delta$ is the usual Kronecker symbol):
\begin{displaymath} \bar{I}=\left(\begin{array}{ccccccc} 0 & 1
& 0 & 0 & 0 & 0 & \ldots \\0 & 0 & 1 & 0 & 0 & 0 & \ldots \\
0 & 0 & 0 & 1
& 0 & 0 & \ldots \\ 0 & 0 & 0 & 0 & 1 & 0 & \ldots \\ 0 & 0 &
0
& 0 & 0 & 1 & \ldots \\0 & 0  & 0 & 0 & 0 & 0 &\ldots\\ \vdots
& \vdots &
\vdots & \vdots & \vdots & \vdots &
\ddots\end{array}\right).\end{displaymath}
We have
\begin{equation}P=A_P^{-1}\bar{I}A_P.\end{equation} Writing
$\overline{A_P}=\bar{I}A_P$, we can write this equation as \begin{equation}P=A_P^{-1}\overline{A_P}.\end{equation} Note that $\overline{A_P}$ is a ``beheaded'' version of $A_P$; that is, it is $A_P$ with the first row removed.

The production matrix $P$ is sometimes \cite{P_W, Shapiro_bij} called the Stieltjes matrix $S_{A_P}$
associated to $A_P$. Other examples of the use of production matrices can be found in \cite{Arregui}, for instance.

\noindent The sequence formed by the row sums of $A_P$ often
has combinatorial significance and is called the sequence
associated to $P$. Its general
term $a_n$ is given by $a_n = u^T P^n e$ where
$$e=\left(\begin{array}{c}1\\1\\1\\\vdots\end{array}\right)$$
In the context of Riordan
arrays, the production matrix associated to a proper Riordan
array takes on a special form\,:

 \begin{proposition} \label{RProdMat}
\cite[Proposition 3.1]{ProdMat}\label{AZ} Let $P$ be
an infinite production matrix and let $A_P$ be the matrix
induced by $P$. Then $A_P$ is an (ordinary) Riordan matrix if
and only if $P$ is
of the form \begin{displaymath} P=\left(\begin{array}{ccccccc}
\xi_0 & \alpha_0 & 0 & 0 & 0 & 0 & \ldots \\\xi_1 & \alpha_1 &
\alpha_0 & 0 &
0 & 0 & \ldots \\ \xi_2 & \alpha_2 & \alpha_1 & \alpha_0 & 0 &
0 & \ldots \\ \xi_3 & \alpha_3 & \alpha_2 & \alpha_1 &
\alpha_0
& 0 & \ldots
\\ \xi_4 & \alpha_4 & \alpha_3 & \alpha_2 & \alpha_1 &
\alpha_0
& \ldots \\\xi_5 & \alpha_5  & \alpha_4 & \alpha_3 & \alpha_2
&
\alpha_1
&\ldots\\ \vdots & \vdots & \vdots & \vdots & \vdots & \vdots
&
\ddots\end{array}\right)\end{displaymath} where $\xi_0 \neq 0$, $\alpha_0 \neq 0$. Moreover, columns $0$
and $1$ of
the matrix $P$ are the $Z$- and $A$-sequences,
respectively, of the Riordan array $A_P$. \end{proposition}
\noindent We recall that we have
$$A(x)=\frac{x}{\bar{f}(x)}, \quad Z(x)=\frac{1}{\bar{f}(x)}\left(1-\frac{1}{g(\bar{f}(x))}\right).$$
\begin{example} We consider
the Riordan array $\mathbf{L}$ where
$$L^{-1}=\left(\frac{1-\lambda x - \mu x^2}{1+ax+bx^2},
\frac{x}{1+ax+bx^2}\right).$$ The production matrix
(Stieltjes matrix) of $$L=\left(\frac{1-\lambda x - \mu
x^2}{1+ax+bx^2}, \frac{x}{1+ax+bx^2}\right)^{-1}$$ is
given by \begin{displaymath}
P=S_L=\left(\begin{array}{ccccccc}
a+\lambda & 1 & 0 & 0 & 0 & 0 & \ldots \\b+\mu & a & 1 & 0 & 0
& 0 & \ldots
\\ 0 & b & a & 1 & 0 & 0 & \ldots \\ 0 & 0 & b & a & 1 & 0 &
\ldots \\ 0 & 0 & 0 & b & a & 1 & \ldots \\0 & 0  & 0 & 0 & b
&
a &\ldots\\
\vdots & \vdots & \vdots & \vdots & \vdots & \vdots &
\ddots\end{array}\right).\end{displaymath} We note that since
\begin{eqnarray*}L^{-1}&=&\left(\frac{1-\lambda x - \mu
x^2}{1+ax+bx^2}, \frac{x}{1+ax+bx^2}\right)\\ &=&(1-\lambda
x-\mu x^2, x)\cdot\left(\frac{1
}{1+ax+bx^2}, \frac{x}{1+ax+bx^2}\right),\end{eqnarray*} we
have $$L=\left(\frac{1-\lambda x - \mu x^2}{1+ax+bx^2},
\frac{x}{1+ax+bx^2}\right)^{-1}=\left(\frac{1 }{1+ax+bx^2},
\frac{x}{1+ax+bx^2}\right)^{-1}\cdot\left(\frac{1}{1-\lambda
x-\mu x^2},x\right).$$
If we now let
$$L_1=\left(\frac{1}{1+ax},\frac{x}{1+ax}\right)\cdot L,$$
then
 \cite{Triple} we obtain that the Stieltjes matrix for
$L_1$ is given by
\begin{displaymath}S_{L_1}=\left(\begin{array}{ccccccc}
\lambda
& 1 & 0 & 0 & 0 & 0 & \ldots \\b+\mu & 0 & 1 & 0 & 0 & 0 &
\ldots \\ 0 & b & 0 & 1 & 0 & 0 & \ldots \\ 0 & 0 & b & 0 & 1
&
0 & \ldots \\ 0 & 0 & 0 & b & 0 & 1 & \ldots \\0 & 0  & 0 & 0
&
b & 0
&\ldots\\ \vdots & \vdots & \vdots & \vdots & \vdots & \vdots
&
\ddots\end{array}\right).\end{displaymath}
\end{example}

\noindent We have in fact the
following
general result \cite{Triple}\,: \begin{proposition} If
$L=(g(x),f(x))$ is a Riordan array and $P=S_L$ is tridiagonal,
then necessarily
\begin{displaymath} P=S_L=\left(\begin{array}{ccccccc} a_1 & 1
& 0 & 0 & 0 & 0 & \ldots \\b_1 & a & 1 & 0 & 0 & 0 & \ldots \\
0 & b & a & 1
& 0 & 0 & \ldots \\ 0 & 0 & b & a & 1 & 0 & \ldots \\ 0 & 0 &
0
& b & a & 1 & \ldots \\0 & 0  & 0 & 0 & b & a &\ldots\\ \vdots
& \vdots &
\vdots & \vdots & \vdots & \vdots &
\ddots\end{array}\right)\end{displaymath} where
$$f(x)=\text{Rev}\frac{x}{1+ax+bx^2}\qquad\text{and}\qquad
g(x)=\frac{1}{1-a_1x-b_1 x f},$$ and vice-versa.
\end{proposition}
\noindent This leads to the
important corollary \begin{corollary}\label{Cor} If $L=(g(x),f(x))$ is a
Riordan array and $P=S_L$ is tridiagonal, with
\begin{equation}\label{P_mat}
P=S_L=\left(\begin{array}{ccccccc} a_1 & 1 & 0 & 0 & 0 & 0 &
\ldots \\b_1 & a & 1 & 0 & 0 & 0 & \ldots \\ 0 & b & a & 1 & 0
& 0 & \ldots \\
0 & 0 & b & a & 1 & 0 & \ldots \\ 0 & 0 & 0 & b & a & 1 &
\ldots \\0 & 0  & 0 & 0 & b & a &\ldots\\ \vdots & \vdots &
\vdots & \vdots &
\vdots & \vdots & \ddots\end{array}\right),\end{equation}
then $L^{-1}$ is the coefficient array of the family of
orthogonal polynomials
$p_n(x)$ where $p_0(x)=1$, $p_1(x)=x-a_1$, and
$$p_{n+1}(x)=(x-a)p_n(x)-b_n p_{n-1}(x), \qquad n \ge 2,$$
where $b_n$ is the sequence
$0,b_1,b,b,b,\ldots$. \end{corollary}

\begin{proof} By Favard's theorem, it suffices to show that $L^{-1}$ defines a family of polynomials $\{p_n(x)\}$ that obey the above three-term recurrence. Now
$L$ is lower-triangular and so $L^{-1}$ is the coefficient array of a family of polynomials $p_n(x)$ (with the degree of $p_n(x)$ being $n$), where
$$ L^{-1} \left(\begin{array}{c} 1\\x\\x^2\\x^3\\\vdots\end{array}\right)=\left(\begin{array}{c} p_0(x)\\p_1(x)\\p_2(x)\\p_3(x)\\\vdots\end{array}\right).$$
\noindent We have $$S_L \cdot L^{-1}=L^{-1}\cdot \bar{L} \cdot L^{-1}=L^{-1}\cdot \bar{I}\cdot L \cdot L^{-1}=L^{-1}\cdot \bar{I}.$$
Thus $$S_L\cdot L^{-1}\cdot (1,x,x^2,\ldots)^T=L^{-1}\cdot \bar{I} \cdot (1,x,x^2,\ldots)^T=L^{-1}\cdot(x,x^2,x^3,\ldots)^T.$$
We therefore obtain
\begin{displaymath} \left(\begin{array}{ccccccc} a_1 & 1 & 0 & 0 & 0 & 0 &
\ldots \\b_1 & a & 1 & 0 & 0 & 0 & \ldots \\ 0 & b & a & 1 & 0
& 0 & \ldots \\
0 & 0 & b & a & 1 & 0 & \ldots \\ 0 & 0 & 0 & b & a & 1 &
\ldots \\0 & 0  & 0 & 0 & b & a &\ldots\\ \vdots & \vdots &
\vdots & \vdots &
\vdots & \vdots & \ddots\end{array}\right)\left(\begin{array}{c} p_0(x)\\p_1(x)\\p_2(x)\\p_3(x)\\\vdots\end{array}\right)=\left(\begin{array}{c} xp_0(x)\\xp_1(x)\\xp_2(x)\\xp_3(x)\\\vdots\end{array}\right),\end{displaymath} from which we infer that
$$p_1(x)=x-a_1,$$ and
$$p_{n+1}(x)+a p_n(x)+b_n p_{n-1}(x)=xp_n(x), \quad n\ge 1,$$ or
$$p_{n+1}(x)=(x-a) p_n(x)-b_n p_{n-1}(x), \quad n \ge 1.$$
\end{proof}
\noindent If we now start with a family of orthogonal polynomials $\{p_n(x)\}$, $p_0(x)=1$, $p_1(x)=x-a_1$, that for $n \ge 1$ obey a three-term recurrence
$$p_{n+1}(x)=(x-a) p_n(x)-b_n p_{n-1}(x), $$
 where $b_n$ is the sequence
$0,b_1,b,b,b,\ldots$, then we can define \cite{P_W} an associated Riordan array $L=(g(x),f(x))$ by

$$f(x)=\text{Rev}\frac{x}{1+ax+bx^2}\qquad\text{and}\qquad
g(x)=\frac{1}{1-a_1x-b_1 x f}.$$

\noindent Clearly, $L^{-1}$ is then the coefficient array of the family of polynomials $\{p_n(x)\}$. Combining these results, we have
\begin{theorem} A Riordan array $L=(g(x), f(x))$ is the inverse of the coefficient array of a family of orthogonal polynomials if and only if its production matrix $P=S_L$ is tri-diagonal.
\end{theorem}
\begin{proof} If $L$ has a tri-diagonal production matrix, then by Corollary (\ref{Cor}), $L^{-1}$ is the coefficient array of orthogonal polynomials. It conversely $L^{-1}$ is the coefficient array of a family of orthogonal polynomials, then using the fact they these polynomials obey a three-term recurrence and the uniqueness of the $Z$- and $A$-sequences, we see using equations (\ref{OP_1}) and (\ref{OP_2}) along with  Proposition \textbf{\ref{AZ}}, that the production matrix is tri-diagonal.
\end{proof}
\begin{proposition} \label{Moments} Let $L=(g(x), f(x))$ be a Riordan array with tri-diagonal production matrix $S_L$. Then
$$[x^n]g(x)=\mathcal{L}(x^n),$$ where $\mathcal{L}$ is the linear functional that defines the associated family of orthogonal polynomials.
\end{proposition}
\begin{proof} Let $L=(l_{i,j})_{i,j \ge 0}$.
We have \cite{Viennot} $$x^n=\sum_{i=0}^n l_{n,i}p_i(x).$$ Applying $\mathcal{L}$, we get
$$\mathcal{L}(x^n)=\mathcal{L}\left(\sum_{i=0}^n l_{n,i}p_i(x)\right)=\sum_{i=0}^n l_{n,i} \mathcal{L}(p_i(x))=\sum_{i=0}^n l_{n,i}\delta_{i,0}=l_{n,0}=[x^n]g(x).$$
\end{proof}
Thus under the conditions of the proposition, by Theorem \ref{CF}, $g(x)$ is the g.f. of the moment sequence $\mu_n=\mathcal{L}(x^n)$. Hence $g(x)$ has the continued fraction expansion
$$g(x)=\cfrac{1}{1-a_1 x-
\cfrac{b_1 x^2}{1-a x -
\cfrac{b x^2}{1-a x -
\cfrac{b x^2}{1-a x -\cdots}}}}.$$

\noindent This can also be established directly. To see this, we use the \begin{lemma} Let $$f(x)=\text{Rev}\frac{x}{1+ax+bx^2}.$$ Then $$\frac{f}{x}=\frac{1}{1-ax-bx^2(f/x)}.$$
\end{lemma}
\begin{proof}
By definition, $f(x)$ is the solution $u(x)$, with $u(0)=0$, of $$\frac{u}{1+au+bu^2}=x.$$ We find
$$u(x)=\frac{1-ax-\sqrt{1-2ax+(a^2-4b)x^2}}{2bx}.$$
Solving the equation
$$v(x)=\frac{1}{1-av-bx^2v},$$  we obtain
$$v(x)=\frac{1-ax-\sqrt{1-2ax+(a^2-4b)x^2}}{2bx^2}=\frac{f(x)}{x}.$$
\end{proof}
\noindent
Thus we have
$$\frac{f(x)}{x}=\frac{1}{x}\text{Rev}\frac{x}{1+ax+bx^2}=
\cfrac{1}{1-ax-
\cfrac{bx^2}{1-ax-
\cfrac{bx^2}{1-\cdots}}}.$$
Now $$g(x)=\frac{1}{1-a_1x-b_1 xf}=\frac{1}{1-a_1 x-b_1 x^2 (f/x)}$$ immediately implies by the above lemma that
$$g(x)=\cfrac{1}{1-a_1x-
\cfrac{b_1x^2}{1-ax-
\cfrac{bx^2}{1-ax-
\cfrac{bx^2}{1-\ldots}}}}.$$

\noindent We note that the elements of the rows of $L^{-1}$ can be identified with the coefficients of the
characteristic polynomials  of the successive principal sub-matrices of $P$.
\begin{example} We consider the Riordan array
$$\left(\frac{1}{1+ax+bx^2}, \frac{x}{1+ax+bx^2}\right).$$
Then
the production matrix
(Stieltjes matrix) of the inverse Riordan array
$\left(\frac{1}{1+ax+bx^2}, \frac{x}{1+ax+bx^2}\right)^{-1}$
left-multiplied by the $k$-th
binomial array
$$\left(\frac{1}{1-kx},\frac{x}{1-kx}\right)=\left(\frac{1}{1-x},\frac{x}{1-x}\right)^k$$
is given by
\begin{displaymath}P=\left(\begin{array}{ccccccc} a+k & 1 & 0
&
0 & 0 & 0 & \ldots \\b & a+k & 1 & 0 & 0 & 0 & \ldots \\ 0 & b
& a+k & 1 & 0
& 0 & \ldots \\ 0 & 0 & b & a+k & 1 & 0 & \ldots \\ 0 & 0 & 0
&
b & a+k & 1 & \ldots \\0 & 0  & 0 & 0 & b & a+k &\ldots\\
\vdots & \vdots &
\vdots & \vdots & \vdots & \vdots &
\ddots\end{array}\right)\end{displaymath} and vice-versa. This
follows since
$$\left(\frac{1}{1+ax+bx^2},\frac{x}{1+ax+bx^2}\right)\cdot\left(\frac{1}{1+kx},\frac{x}{1+kx}\right)=\left(\frac{1}{1+(a+k)x+bx^2},\frac{x}{1+(a+k)x+bx^2}\right).$$
In fact we have the more general result\,:
\begin{displaymath}\begin{split} \left(\frac{1+\lambda x+\mu
x^2}{1+ax+bx^2},\frac{x}{1+ax+bx^2}\right)\cdot\left(\frac{1}{1+kx},\frac{x}{1+kx}\right)=\\
\left(\frac{1+\lambda x+ \mu
x^2}{1+(a+k)x+bx^2},\frac{x}{1+(a+k)x+bx^2}\right).\end{split}\end{displaymath}
The inverse of this last matrix therefore has production
array \begin{displaymath} \left(\begin{array}{ccccccc}
a+k-\lambda & 1 & 0 & 0 & 0 & 0 & \ldots \\b-\mu & a+k & 1 & 0
& 0 & 0 & \ldots \\ 0
& b & a+k & 1 & 0 & 0 & \ldots \\ 0 & 0 & b & a+k & 1 & 0 &
\ldots \\ 0 & 0 & 0 & b & a+k & 1 & \ldots \\0 & 0  & 0 & 0 &
b
& a+k &\ldots\\
\vdots & \vdots & \vdots & \vdots & \vdots & \vdots &
\ddots\end{array}\right).\end{displaymath} \end{example}

\section{Exponential Riordan arrays}
 The \emph{exponential Riordan group} \cite
{PasTri, DeutschShap, ProdMat}, is a set of
infinite lower-triangular integer matrices, where each matrix
is defined by a pair
of generating functions $g(x)=g_0+g_1x+g_2x^2+\cdots$ and
$f(x)=f_1x+f_2x^2+\cdots$ where $g_0 \ne 0$ and $f_1\ne 0$. In what follows, we shall assume
$$g_0=f_1=1.$$
The associated
matrix is the matrix
whose $i$-th column has exponential generating function
$g(x)f(x)^i/i!$ (the first column being indexed by 0). The
matrix corresponding to
the pair $f, g$ is denoted by $[g, f]$.  The group law is given by \begin{displaymath}
[g,
f]\cdot [h,
l]=[g(h\circ f), l\circ f].\end{displaymath} The identity for
this law is $I=[1,x]$ and the inverse of $[g, f]$ is $[g,
f]^{-1}=[1/(g\circ
\bar{f}), \bar{f}]$ where $\bar{f}$ is the compositional
inverse of $f$. We use the notation $\mathit{e}\mathcal{R}$ to
denote this group.

If $\mathbf{M}$ is the matrix $[g,f]$, and
$\mathbf{u}=(u_n)_{n \ge 0}$
is an integer sequence with exponential generating function
$\mathcal{U}$
$(x)$, then the sequence $\mathbf{M}\mathbf{u}$ has
exponential
generating function $g(x)\mathcal{U}(f(x))$. Thus the row sums
of the array
$[g,f]$ have exponential generating function given by $g(x)e^{f(x)}$ since the sequence
$1,1,1,\ldots$ has exponential generating function $e^x$.

As an element of the group of exponential Riordan arrays, the Binomial matrix $\mathbf{B}$ is given by
 $\mathbf{B}=[e^x,x]$. By the above, the exponential
generating function of
its row sums is given by $e^x e^x = e^{2x}$, as expected
($e^{2x}$ is the e.g.f. of $2^n$).

\begin{example} We consider the exponential Riordan array
$[\frac{1}{1-x},x]$, \seqnum{A094587}. This array has elements
\begin{displaymath}\left(\begin{array}{ccccccc} 1 & 0 & 0 & 0
&
0 & 0 & \ldots \\1 & 1 & 0 & 0 & 0 & 0 & \ldots \\ 2 & 2 & 1 &
0 & 0 & 0 &
\ldots \\ 6 & 6 & 3 & 1 & 0 & 0 & \ldots \\ 24 & 24 & 12 & 4 &
1 & 0 & \ldots \\120 & 120  & 60 & 20 & 5 & 1 &\ldots\\ \vdots
& \vdots &
\vdots & \vdots & \vdots & \vdots &
\ddots\end{array}\right)\end{displaymath} and general term $[k
\le n] \frac{n!}{k!}$, and inverse
\begin{displaymath}\left(\begin{array}{ccccccc} 1 & 0 & 0 & 0
&
0 & 0 & \ldots \\-1 & 1 & 0 & 0 & 0 & 0 & \ldots \\ 0 & -2 & 1
& 0 & 0 & 0 &
\ldots \\ 0 & 0 & -3 & 1 & 0 & 0 & \ldots \\ 0 & 0 & 0& -4 & 1
& 0 & \ldots \\0 & 0  & 0 & 0 & -5 & 1 &\ldots\\ \vdots &
\vdots & \vdots &
\vdots & \vdots & \vdots &
\ddots\end{array}\right)\end{displaymath} which is the array
$[1-x,x]$. In particular, we note that the row sums
of the inverse, which begin $1,0,-1,-2,-3,\ldots$ (that is,
$1-n$), have e.g.f. $(1-x)\exp(x)$. This sequence is thus the
binomial transform
of the sequence with e.g.f. $(1-x)$ (which is the sequence
starting $1,-1,0,0,0,\ldots$). \end{example}

\begin{example}
We
consider the
exponential Riordan array $L=[1, \frac{x}{1-x}]$. The general
term of this matrix may be calculated as follows
\begin{eqnarray*}T_{n,k}&=&\frac{n!}{k!}[x^n]\frac{x^k}{(1-x)^k}\\
&=&\frac{n!}{k!}[x^{n-k}](1-x)^{-k}\\
&=&\frac{n!}{k!}[x^{n-k}]\sum_{j=0}^{\infty}\binom{-k}{j}(-1)^jx^j\\
&=&\frac{n!}{k!}[x^{n-k}]\sum_{j=0}^{\infty}\binom{k+j-1}{j}x^j\\
&=&\frac{n!}{k!}\binom{k+n-k-1}{n-k}\\
&=&\frac{n!}{k!}\binom{n-1}{n-k}.\end{eqnarray*} Thus its row
sums, which have e.g.f. $\exp
\left(\frac{x}{1-x}\right)$, have general term $\sum_{k=0}^n
\frac{n!}{k!}\binom{n-1}{n-k}$. This is \seqnum{A000262}, the
`number of
``sets of lists": the number of partitions of $\{1,..,n\}$
into
any number of lists, where a list means an ordered subset'.
Its
general term
is equal to $(n-1)!L_{n-1}(1,-1)$.
\end{example}
We will use the following 
\cite{DeutschShap, ProdMat},important result concerning
matrices that are production matrices for exponential Riordan
arrays. \begin{proposition} Let $A=\left(a_{n,k}\right)_{n,k
\ge
0}=[g(x),f(x)]$ be an exponential Riordan array and let
\begin{equation}\label{seq_def} c(y)=c_0 + c_1 y +c_2 y^2 +
\ldots, \qquad r(y)=r_0
+ r_1 y + r_2 y^2 + \ldots\end{equation} be two formal power
series that that \begin{eqnarray}\label{r_def} r(f(x))&=&f'(x)
\\
 \label{c_def} c(f(x))&=&\frac{g'(x)}{g(x)}. \end{eqnarray} Then
\begin{eqnarray} (i)\qquad a_{n+1,0}&=&\sum_{i} i! c_i a_{n,i}
\\ (ii)\qquad a_{n+1,k}&=& r_0 a_{n,k-1}+\frac{1}{k!}
\sum_{i\ge
k}i!(c_{i-k}+k r_{i-k+1})a_{n,i} \end{eqnarray}  or, assuming $c_k=0$ for $k<0$ and $r_k=0$ for $k<0$,
 \begin{equation}\label{array_def}
a_{n+1,k}=\frac{1}{k!}\sum_{i\ge
k-1} i!(c_{i-k}+k r_{i-k+1})a_{n,i}.\end{equation} Conversely,
starting from the sequences defined by (\ref{seq_def}), the
infinite array
$\left(a_{n,k}\right)_{n,k\ge 0}$ defined by (\ref{array_def})
is
an exponential Riordan array. \end{proposition} \noindent A
consequence of
this proposition is that the production matrix  $P=\left(p_{i,j}\right)_{i,j\ge 0}$ for an exponential Riordan array obtained as in
the
proposition satisfies \cite{ProdMat}
 $$p_{i,j}=\frac{i!}{j!}(c_{i-j}+jr_{i-j+1})  \qquad
(c_{-1}=0).$$
Furthermore, the bivariate exponential function
$$\phi_P(t,z)=\sum_{n,k} p_{n,k}t^k \frac{z^n}{n!}$$ of the
matrix $P$ is given by
$$\phi_P(t,z) = e^{tz}(c(z)+t r(z)).$$
\noindent Note in particular that we have \begin{equation}\label{r} r(x)=f'(\bar{f}(x)),\end{equation} and \begin{equation} \label{c} c(x)=\frac{g'(\bar{f}(x))}{g(\bar{f}(x))}.\end{equation}

\begin{example}
The production matrix of $\left[1,\frac{x}{1+x}\right]$ \seqnum{A111596} is given by
\begin{displaymath}\left(\begin{array}{ccccccc} 0 & 1 & 0 & 0
&
0 & 0 & \ldots \\0 & -2
& 1 & 0 & 0 & 0 & \ldots \\ 0 & 2 & -4 & 1 & 0 & 0 & \ldots \\
0
& 0 & 6 & -6 & 1 & 0 & \ldots \\0 & 0 & 0 & 12 & -8 & 1 &
\ldots \\0 & 0
& 0 & 0 & 20 & -10 &\ldots\\ \vdots & \vdots & \vdots &
\vdots
& \vdots & \vdots & \ddots\end{array}\right).\end{displaymath}
\noindent  The
row sums of $L^{-1}$
 have e.g.f. $\exp\left(\frac{x}{1+x}\right)$, and
start $1, 1, -1, 1, 1, -19, 151, \ldots$. This is
\seqnum{A111884}.
This follows since we have $g(x)=1$ and so $g'(x)=0$, implying that $c(x)=0$, and
$f(x)=\frac{x}{1+x}$ which gives us $\bar{f}(x)=\frac{x}{1-x}$ and
$f'(x)=\frac{1}{(1+x)^2}$. Thus $f'(\bar{f}(x))=c(x)=(1-x)^2$. Hence the bivariate generating function of
$P$ is $e^{xy}(1-x)^2y$, as required.
\end{example}
\begin{example} The
exponential Riordan array
$\mathbf{A}= \left[\frac{1}{1-x},\frac{x}{1-x}\right]$, or
\begin{displaymath}\left(\begin{array}{ccccccc} 1 & 0 & 0 & 0
&
0 & 0 & \ldots
\\1 & 1 & 0 & 0 & 0 & 0 & \ldots \\ 2 & 4 & 1 & 0 & 0 & 0 &
\ldots \\ 6 & 18 & 9 & 1 & 0 & 0 & \ldots \\ 24 & 96 & 72 & 16
& 1 & 0 & \ldots
\\120 & 600  & 600 & 200 & 25 & 1 &\ldots\\ \vdots & \vdots &
\vdots & \vdots & \vdots & \vdots &
\ddots\end{array}\right)\end{displaymath}
has general term $$ T_{n,k}=\frac{n!}{k!}\binom{n}{k}.$$ It is
closely related to the Laguerre polynomials. Its inverse $\mathbf{A}^{-1}$ is
the exponential Riordan array $\left[\frac{1}{1+x},\frac{x}{1+x}\right]$ with general term
$(-1)^{n-k}\frac{n!}{k!}\binom{n}{k}$. This is
\seqnum{A021009}, the triangle
of coefficients of the Laguerre polynomials $L_n(x)$.

The production matrix of the matrix $\mathbf{A}^{-1}=\left[\frac{1}{1+x},\frac{x}{1+x}\right]$  is given by
\begin{displaymath}\left(\begin{array}{ccccccc} 1 & 1 & 0 & 0
&
0 & 0 & \ldots \\1 & 3
& 1 & 0 & 0 & 0 & \ldots \\ 0 & 4 & 5 & 1 & 0 & 0 & \ldots \\
0
& 0 & 9 & 7 & 1 & 0 & \ldots \\0 & 0 & 0 & 16 & 9 & 1 &
\ldots \\0 & 0
& 0 & 0 & 25 & 11 &\ldots\\ \vdots & \vdots & \vdots &
\vdots
& \vdots & \vdots & \ddots\end{array}\right).\end{displaymath}
This follows since we have $g(x)=\frac{1}{1-x}$, and so $g'(x)=\frac{1}{(1-x)^2}$, and
$f(x)=\frac{x}{1-x}$ which yields $\bar{f}(x)=\frac{x}{1+x}$ and $f'(x)=\frac{1}{(1-x)^2}$.
Then $$c(x)=\frac{g'(\bar{f}(x))}{g(\bar{f}(x))}=1+x$$ while
$$r(x)=f'(\bar{f}(x))=(1+x)^2.$$ Thus the bivariate generating function of $P$ is given by
$$ e^{xy}(1+x+(1+x)^2y).$$
\noindent We note
that $$ \mathbf{A}=\exp(\mathbf{S}),$$
where
\begin{displaymath}\mathbf{S}=\left(\begin{array}{ccccccc} 0 &
0 & 0 & 0 & 0 & 0 & \ldots \\1 & 0 & 0 & 0 & 0 & 0 & \ldots \\
0 & 4 & 0 & 0
& 0 & 0 & \ldots \\ 0 & 0 & 9 & 0 & 0 & 0 & \ldots \\ 0 & 0 &
0
& 16 & 0 & 0 & \ldots \\0 & 0  & 0 & 0 & 25 & 0 &\ldots\\
\vdots & \vdots &
\vdots & \vdots & \vdots & \vdots &
\ddots\end{array}\right).\end{displaymath} \end{example}
\begin{example} The exponential Riordan array
$\left[e^x,\ln\left(\frac{1}{1-x}\right)\right]$, or
\begin{displaymath}\left(\begin{array}{ccccccc} 1 & 0 & 0 & 0
&
0 & 0 & \ldots \\1 & 1
& 0 & 0 & 0 & 0 & \ldots \\ 1 & 3 & 1 & 0 & 0 & 0 & \ldots \\
1
& 8 & 6 & 1 & 0 & 0 & \ldots \\ 1 & 24 & 29 & 10 & 1 & 0 &
\ldots \\1 & 89
& 145 & 75 & 15 & 1 &\ldots\\ \vdots & \vdots & \vdots &
\vdots
& \vdots & \vdots & \ddots\end{array}\right)\end{displaymath}
is the
coefficient array for the polynomials $$_2F_0(-n,x;-1)$$ which
are an unsigned version of the Charlier polynomials (of order
$0$)
\cite{wgautschi, Roman, Szego}. This is \seqnum{A094816}. It is equal to
$$[e^x,x]\cdot \left[
1, \ln\left(\frac{1}{1-x}\right)\right],$$ or the product of
the binomial
array $\mathbf{B}$ and the array of (unsigned) Stirling
numbers
of the first kind.
The production matrix of the inverse of this matrix is given by
\begin{displaymath}\left(\begin{array}{ccccccc} -1 & 1 & 0 & 0
&
0 & 0 & \ldots \\1 & -2
& 1 & 0 & 0 & 0 & \ldots \\ 0 & 2 & -3 & 1 & 0 & 0 & \ldots \\
0
& 0 & 3 & -4 & 1 & 0 & \ldots \\0 & 0 & 0 & 4 & -5 & 1 &
\ldots \\0 & 0
& 0 & 0 & 5 & -6 &\ldots\\ \vdots & \vdots & \vdots &
\vdots
& \vdots & \vdots & \ddots\end{array}\right)\end{displaymath}
which indicates the orthogonal nature of these polynomials. We can prove this as follows. We have
$$\left[e^x, \ln\left(\frac{1}{1-x}\right)\right]^{-1}=\left[e^{-(1-e^{-x})},1-e^{-x}\right].$$
Hence $g(x)=e^{-(1-e^{-x})}$ and $f(x)=1-e^{-x}$. We are thus led to the equations
\begin{eqnarray*} r(1-e^{-x})&=&\,e^{-x},\\
c(1-e^{-x})&=&-e^{-x},\end{eqnarray*} with solutions $r(x)=1-x$, $c(x)=x-1$. Thus the bi-variate generating function for the
production matrix of the inverse array is
$$e^{tz}(z-1+t(1-z)),$$ which is what is required.
\end{example}

We can infer the following result from the article \cite{P_W} by Peart and Woan.
\begin{proposition} If $L=[g(x),f(x)]$ is an exponential Riordan array and $P=S_L$ is tridiagonal, then necessarily
\begin{displaymath} P=S_L=\left(\begin{array}{ccccccc} \alpha_0 & 1
& 0 & 0 & 0 & 0 & \ldots \\\beta_1 & \alpha_1 & 1 & 0 & 0 & 0 & \ldots \\
0 & \beta_2 & \alpha_2 & 1
& 0 & 0 & \ldots \\ 0 & 0 & \beta_3 & \alpha_3 & 1 & 0 & \ldots \\ 0 & 0 &
0
& \beta_4 & \alpha_4 & 1 & \ldots \\0 & 0  & 0 & 0 & \beta_5 & \alpha_5 &\ldots\\ \vdots
& \vdots &
\vdots & \vdots & \vdots & \vdots &
\ddots\end{array}\right)\end{displaymath} where $\{\alpha_i\}_{i \ge 0}$ is an arithmetic sequence with common difference $\alpha$,
$\left\{\frac{\beta_i}{i}\right\}_{i\ge 1}$ is an arithmetic sequence with common difference $\beta$, and
$$\ln(g)=\int (\alpha_0+\beta_1 f)\,dx, \quad g(0)=1,$$ where $f$ is given by
$$f'=1 +\alpha f + \beta f^2, \quad f(0)=0,$$
and vice-versa.
\end{proposition}
In the above, we note that $\alpha_0=a_1$ where $g(x)$ is the g.f. of $a_0=1,a_1,a_2,\ldots$.
We have the important
\begin{corollary} If $L=[g(x),f(x)]$ is an exponential Riordan array and $P=S_L$ is tridiagonal, with
\begin{displaymath} P=S_L=\left(\begin{array}{ccccccc} \alpha_0 & 1
& 0 & 0 & 0 & 0 & \ldots \\\beta_1 & \alpha_1 & 1 & 0 & 0 & 0 & \ldots \\
0 & \beta_2 & \alpha_2 & 1
& 0 & 0 & \ldots \\ 0 & 0 & \beta_3 & \alpha_3 & 1 & 0 & \ldots \\ 0 & 0 &
0
& \beta_4 & \alpha_4 & 1 & \ldots \\0 & 0  & 0 & 0 & \beta_5 & \alpha_5 &\ldots\\ \vdots
& \vdots &
\vdots & \vdots & \vdots & \vdots &
\ddots\end{array}\right),\end{displaymath}
then $L^{-1}$ is the coefficient array of the family of monic orthogonal polynomials $p_n(x)$ where
$p_0(x)=1$, $p_1(x)=x-a_1=x-\alpha_0$, and
$$p_{n+1}(x)=(x-\alpha_n)p_n(x)-\beta_n p_{n-1}(x), \quad n \ge 0.$$
\end{corollary}
\begin{proof} By Favard's theorem, it suffices to show that $L^{-1}$ defines a family of polynomials $\{p_n(x)\}$ that obey the above three-term recurrence. Now
$L$ is lower-triangular and so $L^{-1}$ is the coefficient array of a family of polynomials $p_n(x)$, where
$$ L^{-1} \left(\begin{array}{c} 1\\x\\x^2\\x^3\\\vdots\end{array}\right)=\left(\begin{array}{c} p_0(x)\\p_1(x)\\p_2(x)\\p_3(x)\\\vdots\end{array}\right).$$
\noindent We have $$S_L \cdot L^{-1}=L^{-1}\cdot \bar{L} \cdot L^{-1}=L^{-1}\cdot \bar{I}\cdot L \cdot L^{-1}=L^{-1}\cdot \bar{I}.$$
Thus $$S_L\cdot L^{-1}\cdot (1,x,x^2,\ldots)^T=L^{-1}\cdot \bar{I} \cdot (1,x,x^2,\ldots)^T=L^{-1}\cdot(x,x^2,x^3,\ldots)^T.$$
We therefore obtain
\begin{displaymath} \left(\begin{array}{ccccccc} \alpha_0 & 1
& 0 & 0 & 0 & 0 & \ldots \\\beta_1 & \alpha_1 & 1 & 0 & 0 & 0 & \ldots \\
0 & \beta_2 & \alpha_2 & 1
& 0 & 0 & \ldots \\ 0 & 0 & \beta_3 & \alpha_3 & 1 & 0 & \ldots \\ 0 & 0 &
0
& \beta_4 & \alpha_4 & 1 & \ldots \\0 & 0  & 0 & 0 & \beta_5 & \alpha_5 &\ldots\\ \vdots
& \vdots &
\vdots & \vdots & \vdots & \vdots &
\ddots\end{array}\right)\left(\begin{array}{c} p_0(x)\\p_1(x)\\p_2(x)\\p_3(x)\\\vdots\end{array}\right)=\left(\begin{array}{c} xp_0(x)\\xp_1(x)\\xp_2(x)\\xp_3(x)\\\vdots\end{array}\right),\end{displaymath} from which we infer
$$p_1(x)=x-\alpha_0,$$ and
$$p_{n+1}(x)+\alpha_n p_n(x)+\beta_n p_{n-1}(x)=xp_n(x), \quad n\ge 1,$$ or
$$p_{n+1}(x)=(x-\alpha_n) p_n(x)-\beta_n p_{n-1}(x) \quad n\ge 1.$$
\end{proof}
\noindent If we now start with a family of orthogonal polynomials $\{p_n(x)\}$ that obeys a three-term recurrence
$$p_{n+1}(x)=(x-\alpha_n) p_n(x)-\beta_n p_{n-1}(x), \quad n\ge 1,$$ with $p_0(x)=1$, $p_1(x)=x-\alpha_0$,
where $\{\alpha_i\}_{i \ge 0}$ is an arithmetic sequence with common difference $\alpha$ and
$\left\{\frac{\beta_i}{i}\right\}_{i\ge 1}$ is an arithmetic sequence with common difference $\beta$, then we can define \cite{P_W} an associated exponential Riordan array $L=[g(x),f(x)]$ by
$$f'=1 +\alpha f + \beta f^2, \quad f(0)=0,$$ and
$$\ln(g)=\int (\alpha_0+\beta_1 f)\,dx, \quad g(0)=1.$$ Clearly, $L^{-1}$ is then the coefficient array of the family of polynomials $\{p_n(x)\}$. Gathering these results, we have
\begin{theorem} An exponential Riordan array $L=[g(x), f(x)]$ is the inverse of the coefficient array of a family of orthogonal polynomials if and only if its production matrix $P=S_L$ is tri-diagonal.
\end{theorem}
\begin{proposition} Let $L=[g(x), f(x)]$ be an exponential Riordan array with tri-diagonal production matrix $S_L$. Then
$$n![x^n]g(x)=\mathcal{L}(x^n)=\mu_n,$$ where $\mathcal{L}$ is the linear functional that defines the associated family of orthogonal polynomials.
\end{proposition}
\begin{proof} Let $L=(l_{i,j})_{i,j \ge 0}$.
We have $$x^n=\sum_{i=0}^n l_{n,i}p_i(x).$$ Applying $\mathcal{L}$, we get
$$\mathcal{L}(x^n)=\mathcal{L}\left(\sum_{i=0}^n l_{n,i}p_i(x)\right)=\sum_{i=0}^n l_{n,i} \mathcal{L}(p_i(x))=\sum_{i=0}^n l_{n,i}\delta_{i,0}=l_{n,0}=n![x^n]g(x).$$
\end{proof}
\begin{corollary} Let $L=[g(x), f(x)]$ be an exponential Riordan array with tri-diagonal production matrix $S_L$. Then the moments $\mu_n$ of the associated family of orthogonal polynomials are given by the terms of the first column of $L$.
\end{corollary}
Thus under the conditions of the proposition, $g(x)$ is the g.f. of the moment sequence $\mathcal{L}(x^n)$. Hence $g(x)$ has the continued fraction expansion
$$g(x)=\cfrac{1}{1-\alpha_0 x-
\cfrac{\beta_1 x^2}{1-\alpha_1 x -
\cfrac{\beta_2 x^2}{1-\alpha_2 x -
\cfrac{\beta_3 x^2}{1-\alpha_3 x -\cdots}}}}.$$

When we come to study Hermite polynomials, we shall be working with elements of the exponential Appell subgroup. By the
\emph{exponential Appell subgroup} $\mathcal{A}\mathit{e}\mathcal{R}$ of
$\mathit{e}\mathcal{R}$ we understand the set of arrays of the
form $[f(x),x]$.

Let $\mathbf{A} \in \mathcal{A}\mathit{e}\mathcal{R}$ correspond to the
sequence $(a_n)_{n \ge 0}$, with e.g.f. $f(x)$. Let $\mathbf{B} \in
\mathcal{A}\mathit{e}\mathcal{R}$ correspond to the sequence $(b_n)$,
with e.g.f. $g(x)$. Then we have \begin{enumerate} \item The
row sums of
$\mathbf{A}$ are the binomial transform of $(a_n)$. \item The
inverse of $\mathbf{A}$ is the sequence array for the sequence
with e.g.f.
$\frac{1}{f(x)}$. \item The product $\mathbf{A}\mathbf{B}$ is
the sequence array for the exponential convolution
$a*b(n)=\sum_{k=0}^n
\binom{n}{k} a_k b_{n-k}$ with e.g.f. $f(x) g(x)$.
\end{enumerate}
\noindent For instance, the row sums of $\mathbf{A}=[f(x),x]$ will have e.g.f. given by
$$ [f(x),x]\cdot e^x= f(x)e^x =e^x f(x)=[e^x,x]\cdot f(x),$$ which is the e.g.f. of the binomial transform of $(a_n)$.
\begin{example} We consider the matrix $[\cosh(x), x]$, \seqnum{A119467}, with
elements \begin{displaymath}\left(\begin{array}{ccccccc} 1 & 0
& 0 & 0 & 0 & 0 &
\ldots \\0 & 1 & 0 & 0 & 0 & 0 & \ldots \\ 1 & 0 & 1 & 0 & 0 &
0 & \ldots \\ 0 & 3 & 0 & 1 & 0 & 0 & \ldots \\ 1 & 0 & 6 & 0
&
1 & 0 &
\ldots \\0 & 5  & 0 & 10 & 0 & 1 &\ldots\\ \vdots & \vdots &
\vdots & \vdots & \vdots & \vdots &
\ddots\end{array}\right).\end{displaymath}
The row sums of this matrix have e.g.f. $\cosh(x)\exp(x)$,
which is the e.g.f. of the sequence $1,1,2,4,8,16,\ldots$. The
inverse matrix is
$[\text{sech}(x),x]$, \seqnum{A119879},  with entries
\begin{displaymath}\left(\begin{array}{ccccccc} 1 & 0 & 0 & 0
&
0 & 0 & \ldots \\0 & 1 & 0 & 0 & 0 & 0 &
\ldots \\ -1 & 0 & 1 & 0 & 0 & 0 & \ldots \\ 0 & -3 & 0 & 1 &
0
& 0 & \ldots \\ 5 & 0 & -6 & 0 & 1 & 0 & \ldots \\0 & 25  & 0
&
-10 & 0 & 1
&\ldots\\ \vdots & \vdots & \vdots & \vdots & \vdots & \vdots
&
\ddots\end{array}\right).\end{displaymath} The row sums of this
matrix have
e.g.f. $\text{sech}(x)\exp(x)$. This is \seqnum{A155585}. \end{example}

\section{The Hankel transform of an integer sequence}
The {\it
Hankel
transform} of a given sequence
$A=\{a_0,a_1,a_2,...\}$ is the
sequence of Hankel determinants $\{h_0, h_1, h_2,\dots \}$
where
$h_{n}=|a_{i+j}|_{i,j=0}^{n}$, i.e

\begin{center} \begin{equation}
 \label{gen1}
 A=\{a_n\}_{n\in\mathbb N_0}\quad \rightarrow \quad
 h=\{h_n\}_{n\in\mathbb N_0}:\quad
h_n=\left| \begin{array}{ccccc}
 a_0\ & a_1\  & \cdots & a_n  &  \\
 a_1\ & a_2\  &        & a_{n+1}  \\
\vdots &      & \ddots &          \\
 a_n\ & a_{n+1}\ &    & a_{2n}
\end{array} \right|. \end{equation} \end{center} The Hankel
transform of a sequence $a_n$ and its binomial transform are
equal.

In the case that $a_n$ has g.f. $g(x)$ expressible in the form
$$g(x)=\cfrac{a_0}{1-\alpha_0 x-
\cfrac{\beta_1 x^2}{1-\alpha_1 x-
\cfrac{\beta_2 x^2}{1-\alpha_2 x-
\cfrac{\beta_3 x^2}{1-\alpha_3 x-\cdots}}}}$$ then
we have \cite{Kratt}
\begin{equation}\label{Kratt} h_n = a_0^{n+1} \beta_1^n\beta_2^{n-1}\cdots \beta_{n-1}^2\beta_n=a_0^{n+1}\prod_{k=1}^n
\beta_k^{n+1-k}.\end{equation}
Note that this independent from $\alpha_n$.

We note that $\alpha_n$ and $\beta_n$ are in general not integers.
\noindent Now let $H\left(\begin{array}{ccc} u_1&\ldots&u _k\\
v_1&\ldots& v_k\end{array}\right)$ be the determinant
of
Hankel type with $(i,j)$-th term $\mu_{u_i+v_j}$. Let
$$\Delta_n=H\left(\begin{array}{cccc} 0&1&\ldots&n\\
0&1&\ldots&n\end{array}\right),\qquad
\Delta'_n=H_n\left(\begin{array}{ccccc}
0&1&\ldots&n-1&n\\ 0&1&\ldots&n-1&n+1\end{array}\right).$$
Then
we have
\begin{equation}\alpha_n=\frac{\Delta'_n}{\Delta_n}-\frac{\Delta'_{n-1}}{\Delta_{n-1}},\qquad
\beta_n=\frac{\Delta_{n-2} \Delta_n}{\Delta_{n-1}^2}.\end{equation}

\section{Legendre polynomials}
We recall that the Legendre polynomials $P_n(x)$ can be defined by
$$P_n(x)=\sum_{k=0}^n (-1)^k \binom{n}{k}^2 \left(\frac{1+x}{2}\right)^{n-k}\left(\frac{1-x}{2}\right)^k.$$
Their generating function is given by
$$\frac{1}{\sqrt{1-2xt +t^2}}=\sum_{n=0}^{\infty} P_n(x)t^n.$$
We note that the production matrix of the inverse of the coefficient array of these polynomials is given by
\begin{displaymath}\left(\begin{array}{ccccccc} 0 & 1 &
0
& 0 & 0 & 0 & \ldots \\\frac{1}{3} & 0 & \frac{2}{3} & 0 & 0 & 0 & \ldots \\ 0 & \frac{2}{5}
& 0 & \frac{3}{5} & 0 &
0 & \ldots \\ 0 & 0 & \frac{3}{7} & 0 & \frac{4}{7} & 0 & \ldots \\ 0 & 0 & 0
& \frac{4}{9} & 0 & \frac{5}{9} & \ldots \\0 & 0 & 0 & 0 & \frac{5}{11} & 0
&\ldots\\
\vdots &
\vdots & \vdots & \vdots & \vdots & \vdots &
\ddots\end{array}\right),\end{displaymath}
which corresponds to the fact that the $P_n(x)$ satisfy the following three-term recurrence
$$(n+1)P_{n+1}(x)=(2n+1)xP_n(x)-nP_{n-1}(x).$$

\noindent The shifted Legendre polynomials $\tilde{P}_n (x)$ are defined by
$$\tilde{P}_n(x)=P_n(2x-1).$$
They satisfy
$$\tilde{P}_n(x)=(-1)^n \sum_{k=0}^n \binom{n}{k}\binom{n+k}{k}(-x)^k = \sum_{k=0}^n (-1)^{n-k}\binom{n+k}{2k}\binom{2k}{k}x^k.$$
Their coefficient array begins
\begin{displaymath}\left(\begin{array}{ccccccc} 1 & 0 &
0
& 0 & 0 & 0 & \ldots \\-1 & 2 & 0 & 0 & 0 & 0 & \ldots \\ 1 & -6
& 6 & 0 & 0 &
0 & \ldots \\ -1 & 12 & -30 & 20 & 0 & 0 & \ldots \\ 1 & -20 & 90
& -140 & 70 & 0 & \ldots \\-1 & 30 & -210 & 560 & -630 & 252
&\ldots\\
\vdots &
\vdots & \vdots & \vdots & \vdots & \vdots &
\ddots\end{array}\right),\end{displaymath} and so the first few terms begin
$$1,2x-1,6x^2-6x+1,20x^3-30x^2+12x-1,\ldots$$
We clearly have
$$\frac{1}{\sqrt{1-2(2x-1)t +t^2}}=\sum_{n=0}^{\infty} \tilde{P}_n(x)t^n.$$
\section{Legendre polynomials as moments}
Our goal in this section is to represent the Legendre polynomials as the first column of a Riordan array whose production matrix is
tri-diagonal. We first of all consider the so-called shifted Legendre polynomials. We have
\begin{proposition}
The inverse $\mathbf{L}$ of the Riordan array
$$\left(\frac{1+r(1-r)x^2}{1+(2r-1)x+r(r-1)x^2},\frac{x}{1+(2r-1)x+r(r-1)x^2}\right)$$ has as its first column
the shifted Legendre polynomials $\tilde{P}_n(r)$. The production matrix of $\mathbf{L}$ is tri-diagonal.
\end{proposition}
\begin{proof}
Indeed, standard Riordan array techniques show that we have
\begin{eqnarray*}
\mathbf{L}&=&\left(\frac{1+r(1-r)x^2}{1+(2r-1)x+r(r-1)x^2},\frac{x}{1+(2r-1)x+r(r-1)x^2}\right)^{-1}\\
&=&\left(\frac{1}{\sqrt{1-2(2r-1)x+x^2}}, \frac{1-(2r-1)x-\sqrt{1-2(2r-1)x+x^2}}{2r(r-1)x}\right).\end{eqnarray*}
This establishes the first part. Now using equations (\ref{AZ_eq}), with
$$f(x)=\frac{1-(2r-1)x-\sqrt{1-2(2r-1)x+x^2}}{2r(r-1)x}, \quad \bar{f}(x)=\frac{x}{1+(2r-1)x+r(r-1)x^2},$$ and $$g(x)=\frac{1}{\sqrt{1-2(2r-1)x+x^2}},$$ we find that
$$Z(x)=(2r-1)+2rx(r-1),\quad A(x)=1+x(2r-1)+x^2r(r-1).$$
Hence by Proposition \ref{RProdMat} the production matrix $P=S_L$ of $\mathbf{L}$ is given by
\begin{displaymath}\left(\begin{array}{ccccccc} 2r-1 & 1 &
0
& 0 & 0 & 0 & \ldots \\2r(r-1) & 2r-1 & 1 & 0 & 0 & 0 & \ldots \\ 0 & r(r-1)
& 2r-1 & 1 & 0 &
0 & \ldots \\ 0 & 0 & r(r-1) & 2r-1 & 1 & 0 & \ldots \\ 0 & 0 & 0
& r(r-1) & 2r-1 & 1 & \ldots \\0 & 0 & 0 & 0 & r(r-1) & 2r-1
&\ldots\\
\vdots &
\vdots & \vdots & \vdots & \vdots & \vdots &
\ddots\end{array}\right).\end{displaymath}
\end{proof}
\begin{corollary}
The shifted Legendre polynomials are moments of the family of orthogonal polynomials whose coefficient array is given by
$$\mathbf{L}^{-1}=\left(\frac{1+r(1-r)x^2}{1+(2r-1)x+r(r-1)x^2},\frac{x}{1+(2r-1)x+r(r-1)x^2}\right).$$
\end{corollary}
\begin{proof} This follows from the above result and Proposition \ref{Moments}.
\end{proof}

\begin{proposition} The Hankel transform of the sequence $\tilde{P}_n(r)$ is given by $2^n (r(r-1))^{\binom{n+1}{2}}$.
\end{proposition}
\begin{proof}
From the above, the g.f. of $\tilde{P}_n(r)$ is given by
$$\cfrac{1}{1-(2r-1)x-
\cfrac{2r(r-1)x^2}{1-(2r-1)x-
\cfrac{r(r-1)x^2}{1-(2r-1)x-
\cfrac{r(r-1)x^2}{1-\cdots}}}}.$$
The result now follows from Equation (\ref{Kratt}).
\end{proof}
\noindent We note that
$$\tilde{P}_n(r)=\frac{1}{\pi}\int_{-2\sqrt{r(r-1)}+2r-1}^{2\sqrt{r(r-1)}+2r-1} \frac{x^n}{\sqrt{-x^2+2(2r-1)x-1}}\, dx$$ gives an
explicit moment representation for $\tilde{P}_n(r)$.

Turning now to the Legendre polynomials $P_n(x)$, we have the following result.
\begin{proposition}
The inverse $\mathbf{L}$ of the Riordan array
$$\left(\frac{1+\frac{1-r^2}{4}x^2}{1+rx+\frac{r^2-1}{2}x^2},\frac{x}{1+rx+\frac{r^2-1}{2}x^2}\right)$$ has as its first column
the  Legendre polynomials $P_n(r)$. The production matrix of $\mathbf{L}$ is tri-diagonal.
\end{proposition}
\begin{proof}
We have
\begin{eqnarray*}
\mathbf{L}&=&\left(\frac{1+\frac{1-r^2}{4}x^2}{1+rx+\frac{r^2-1}{2}x^2},\frac{x}{1+rx+\frac{r^2-1}{2}x^2}\right)^{-1}\\
&=&\left(\frac{1}{\sqrt{1-2rx+x^2}}, \frac{2(1-rx-\sqrt{1-2rx+x^2})}{x(r^2-1)}\right).\end{eqnarray*} This proves the first assertion.
Using equations (\ref{AZ}) again, we obtain the following tri-diagonal matrix as the production matrix of $\mathbf{L}$:
\begin{displaymath}\left(\begin{array}{ccccccc} r & 1 &
0
& 0 & 0 & 0 & \ldots \\\frac{r^2-1}{2} & r & 1 & 0 & 0 & 0 & \ldots \\ 0 & \frac{r^2-1}{4}
& r & 1 & 0 &
0 & \ldots \\ 0 & 0 & \frac{r^2-1}{4} & r & 1 & 0 & \ldots \\ 0 & 0 & 0
& \frac{r^2-1}{4} & r & 1 & \ldots \\0 & 0 & 0 & 0 & \frac{r^2-1}{4} & r
&\ldots\\
\vdots &
\vdots & \vdots & \vdots & \vdots & \vdots &
\ddots\end{array}\right).\end{displaymath}
\end{proof}
\begin{corollary}
$$\mathbf{L}^{-1}=\left(\frac{1+\frac{1-r^2}{4}x^2}{1+rx+\frac{r^2-1}{2}x^2},\frac{x}{1+rx+\frac{r^2-1}{2}x^2}\right)$$ is the
coefficient array of a set of orthogonal polynomials for which the Legendre polynomials are moments.
\end{corollary}
\begin{proposition} The Hankel transform of $P_n(r)$ is given by $$\frac{(r^2-1)^{\binom{n+1}{2}}}{2^{n^2}}.$$
\end{proposition}
\begin{proof}
From the above, we obtain that the g.f. of $P_n(r)$ can be expressed as
$$\cfrac{1}{1-rx-
\cfrac{\frac{r^2-1}{2}x^2}{1-rx-
\cfrac{\frac{r^2-1}{4}x^2}{1-rx-
\cfrac{\frac{r^2-1}{4}x^2}{1-\cdots}}}}.$$ The result now follows from Equation (\ref{Kratt}).
\end{proof}
\noindent We end this section by noting that
$$P_n(r)=\frac{1}{\pi}\int_{r-\sqrt{r^2-1}}^{r+\sqrt{r^2-1}} \frac{x^n}{\sqrt{-x^2+2rx-1}}\, dx$$ gives an explicit moment
representation for $P_n(r)$.
\section{Hermite polynomials}
The Hermite polynomials may be defined as
$$H_n(x)=\sum_{k=0}^{\lfloor \frac{n}{2} \rfloor}\frac{(-1)^k (2x)^{n-2k}}{k!(n-2k)!}.$$
The generating function for $H_n(x)$ is given by
$$e^{2xt-t^2}=\sum_{n=0}^{\infty} H_n(x)\frac{t^n}{n!}.$$
The unitary Hermite polynomials (also called normalized Hermite polynomials) are given by
$$He_n(x)=2^{-\frac{n}{2}} H_n (\sqrt{2}x)=\sum_{k=0}^n
\frac{n!}{(-2)^{\frac{n-k}{2}}k!\left(\frac{n-k}{2}\right)!}\frac{1+(-1)^{n-k}}{2}x^k.$$
Their generating function is given by
$$e^{xt-\frac{t^2}{2}}=\sum_{n=0}^{\infty} He_n(x)\frac{t^n}{n!}.$$
We note that the coefficient array of $He_n$ is a proper exponential Riordan array, equal to
$$\left[ e^{-\frac{x^2}{2}},x\right].$$ This array \seqnum{A066325} begins
\begin{displaymath}\left(\begin{array}{ccccccc} 1 & 0 &
0
& 0 & 0 & 0 & \ldots \\0 & 1 & 0 & 0 & 0 & 0 & \ldots \\ -1 & 0
& 1 & 0 & 0 &
0 & \ldots \\ 0 & -3 & 0 & 1 & 0 & 0 & \ldots \\ 3 & 0 & -6
& 0 & 1 & 0 & \ldots \\0 & 15 & 0 & -10 & 0 & 1
&\ldots\\
\vdots &
\vdots & \vdots & \vdots & \vdots & \vdots &
\ddots\end{array}\right).\end{displaymath} It is the aeration of the alternating sign version of the
Bessel coefficient array \seqnum{A001497}. The inverse of this matrix has production matrix
\begin{displaymath}\left(\begin{array}{ccccccc} 0 & 1 &
0
& 0 & 0 & 0 & \ldots \\1 & 0 & 1 & 0 & 0 & 0 & \ldots \\ 0 & 2
& 0 & 1 & 0 &
0 & \ldots \\ 0 & 0 & 3 & 0 & 1 & 0 & \ldots \\ 0 & 0 & 0
& 4 & 0 & 1 & \ldots \\0 & 0 & 0 & 0 & 5 & 0
&\ldots\\
\vdots &
\vdots & \vdots & \vdots & \vdots & \vdots &
\ddots\end{array}\right),\end{displaymath} which corresponds to the fact that we have the following three-term recurrence for $He_n$:
$$He_{n+1}(x)=xHe_n(x)-n He_{n-1}(x).$$
\section{Hermite polynomials as moments}

\begin{proposition} The proper exponential Riordan array
$$\mathbf{L}=\left[e^{rx-\frac{x^2}{2}},x\right]$$ has as first column the unitary Hermite polynomials $He_n(r)$. This array has a tri-diagonal
production array.
\end{proposition}
\begin{proof} The first column of $\mathbf{L}$ has generating function $e^{rx-\frac{x^2}{2}}$, from which the first assertion follows.
We now use equations (\ref{r}) and (\ref{c}) to calculate the production matrix $P=S_L$.
We have $f(x)=x$, so that $\bar{f}(x)=x$ and $f'(\bar{f}(x))=1=r(x)$. Also,
 $g(x)=e^{rx-\frac{x^2}{2}}$ implies that $g'(x)=g(x)(r-x)$, and hence
$$c(x)=\frac{g'(\bar{f}(x))}{g(\bar{f}(x))}=r-x.$$
Thus the production array of $\mathbf{L}$ is indeed tri-diagonal, beginning
\begin{displaymath}\left(\begin{array}{ccccccc} r & 1 &
0
& 0 & 0 & 0 & \ldots \\-1 & r & 1 & 0 & 0 & 0 & \ldots \\ 0 & -2
& r & 1 & 0 &
0 & \ldots \\ 0 & 0 & -3 & r & 1 & 0 & \ldots \\ 0 & 0 & 0
& -4 & r & 1 & \ldots \\0 & 0 & 0 & 0 & -5 & r
&\ldots\\
\vdots &
\vdots & \vdots & \vdots & \vdots & \vdots &
\ddots\end{array}\right).\end{displaymath}
\end{proof}

We note that $\mathbf{L}$ starts
\begin{displaymath}\left(\begin{array}{ccccccc} 1 & 0 &
0
& 0 & 0 & 0 & \ldots \\r & 1 & 0 & 0 & 0 & 0 & \ldots \\ r^2-1 & 2r
& 1 & 0 & 0 &
0 & \ldots \\ r(r^2-3) & 3(r^2-1) & 3r & 1 & 0 & 0 & \ldots \\ r^4-6r^2+3  & 4r(r^2-3) & 6(r^2-1)
& 4r & 1 & 0 & \ldots \\r(r^4-10r^2+15) & 5(r^4-6r^2+3) & 10r(r^2-3) & 10(r^2-1) & 5r & 1
&\ldots\\
\vdots &
\vdots & \vdots & \vdots & \vdots & \vdots &
\ddots\end{array}\right).\end{displaymath}
Thus $$\mathbf{L}^{-1}=\left[e^{-rx+\frac{x^2}{2}},x\right]$$ is the coefficient array of a set of orthogonal polynomials which have as moments
the unitary Hermite polynomials. These new orthogonal polynomials satisfy the three-term recurrence
$$\mathfrak{H}_{n+1}(x)=(x-r)\mathfrak{H}_n(x)+n \mathfrak{H}_{n-1}(x),$$ with
$\mathfrak{H}_0=1$, $\mathfrak{H}_1=x-r$.

\begin{proposition} The Hankel transform of the sequence $He_n(r)$ is given by $(-1)^{\binom{n+1}{2}}\prod_{k=0}^n  k!$
\end{proposition}
\begin{proof}
By the above, the g.f. of $H_n(r)$ is given by
$$
\cfrac{1}{1-rx+
\cfrac{2x^2}{1-rx+
\cfrac{3x^2}{1-rx+
\cfrac{4x^2}{1-\ldots}}}}. $$
The result now follows from Equation (\ref{Kratt}). \end{proof}

Turning now to the Hermite polynomials $H_n(x)$, we have the following result.
\begin{proposition} The proper exponential Riordan array
$$\mathbf{L}=\left[e^{2rx-x^2},x\right]$$ has as first column the  Hermite polynomials $H_n(r)$. This array has a tri-diagonal
production array.
\end{proposition}
\begin{proof} The first column of $\mathbf{L}$ has generating function $e^{2rx-x^2}$, from which the first assertion
follows.
Using equations (\ref{r}) and (\ref{c}) we find that the production array of $\mathbf{L}$ is indeed tri-diagonal, beginning
\begin{displaymath}\left(\begin{array}{ccccccc} 2r & 1 &
0
& 0 & 0 & 0 & \ldots \\-2 & 2r & 1 & 0 & 0 & 0 & \ldots \\ 0 & -4
& 2r & 1 & 0 &
0 & \ldots \\ 0 & 0 & -6 & 2r & 1 & 0 & \ldots \\ 0 & 0 & 0
& -8 & 2r & 1 & \ldots \\0 & 0 & 0 & 0 & -10 & 2r
&\ldots\\
\vdots &
\vdots & \vdots & \vdots & \vdots & \vdots &
\ddots\end{array}\right).\end{displaymath}
\end{proof}

We note that $\mathbf{L}$ starts
\begin{displaymath}\left(\begin{array}{ccccccc} 1 & 0 &
0
& 0 & 0 & 0 & \ldots \\2r & 1 & 0 & 0 & 0 & 0 & \ldots \\ 2(2r^2-1) & 4r
& 1 & 0 & 0 &
0 & \ldots \\ 4r(2r^2-3) & 6(2r^2-1) & 6r & 1 & 0 & 0 & \ldots \\ 4(4r^3-12r^2+3)  & 16r(2r^2-3) & 12(2r^2-1)
& 8r & 1 & 0 & \ldots \\8r(4r^4-20r^2+15) & 20(4r^4-12r^2+3) & 40r(2r^2-3) & 20(2r^2-1) & 10r & 1
&\ldots\\
\vdots &
\vdots & \vdots & \vdots & \vdots & \vdots &
\ddots\end{array}\right).\end{displaymath}
Thus $$\mathbf{L}^{-1}=\left[e^{-2rx+x^2},x\right]$$ is the coefficient array of a set of orthogonal polynomials which have as moments
the Hermite polynomials. These new orthogonal polynomials satisfy the three-term recurrence
$$\mathfrak{H}_{n+1}(x)=(x-2r)\mathfrak{H}_n(x)+2n \mathfrak{H}_{n-1}(x),$$ with
$\mathfrak{H}_0=1$, $\mathfrak{H}_1=x-2r$.
\begin{proposition} The Hankel transform of the sequence $H_n(r)$ is given by $(-1)^{\binom{n+1}{2}}\prod_{k=0}^n 2^k k!$
\end{proposition}
\begin{proof}
By the above, the g.f. of $H_n(r)$ is given by
$$
\cfrac{1}{1-2rx+
\cfrac{2x^2}{1-2rx+
\cfrac{4x^2}{1-2rx+
\cfrac{6x^2}{1-\ldots}}}}. $$
The result now follows from Equation (\ref{Kratt}). \end{proof}

\section{Acknowledgements} The author would like to thank an anonymous referee for their careful reading and cogent suggestions which have hopefully led to a clearer paper.

\section{Appendix - The Stieltjes transform of a measure} The
\emph{Stieltjes transform} of a measure $\mu$ on $\mathbb{R}$
is a function $G_{\mu}$
defined on $\mathbb{C}\setminus \mathbb{R}$ by
$$G_{\mu}(z)=\int_{\mathbb{R}}\frac{1}{z-t}\mu(t).$$ If $f$ is
a bounded continuous function on
$\mathbb{R}$, we have $$\int_{\mathbb{R}}f(x)\mu(x)=-\lim_{y
\to 0^{+}}\int_{\mathbb{R}}f(x)\Im G_{\mu}(x+iy)dx.$$ If $\mu$
has compact
support, then $G_{\mu}$ is holomorphic at infinity and for
large $z$, $$G_{\mu}(z)=\sum_{n=0}^{\infty}
\frac{a_n}{z^{n+1}},$$ where
$a_n=\int_{\mathbb{R}}t^n \mu(t)$ are the moments of the
measure. If $\mu(t)=d\psi(t)=\psi'(t)dt$ then (Stieltjes-Perron)
$$\psi(t)-\psi(t_0)=-\frac{1}{\pi}\lim_{y \to
0^{+}}\int_{t_0}^t \Im G_{\mu}(x+iy)dx.$$ If now $g(x)$ is the
generating function of a
sequence $a_n$, with $g(x)=\sum_{n=0}^{\infty}a_n x^n$, then
we
can define
$$G(z)=\frac{1}{z}g\left(\frac{1}{z}\right)=\sum_{n=0}^{\infty}\frac{a_n}{z^{n+1}}.$$
By this means, under the right circumstances we can
retrieve the density function for the measure that defines the
elements $a_n$ as moments.
\begin{example}
We let $g(z)=\frac{1-\sqrt{1-4z}}{2z}$ be the g.f. of the Catalan numbers. Then
$$G(z)=\frac{1}{z}g\left(\frac{1}{z}\right)=\frac{1}{2}\left(1-\sqrt{\frac{x-4}{x}}\right).$$
Then $$\Im G_{\mu}(x+iy)=-\frac{\sqrt{2}\sqrt{\sqrt{x^2+y^2}\sqrt{x^2-8x+y^2+16}-x^2+4x-y^2}}{4\sqrt{x^2+y^2}},$$
and so we obtain \begin{eqnarray*}\psi'(x)&=&-\frac{1}{\pi}\lim_{y \to
0^{+}}\left\{-\frac{\sqrt{2}\sqrt{\sqrt{x^2+y^2}\sqrt{x^2-8x+y^2+16}-x^2+4x-y^2}}{4\sqrt{x^2+y^2}}\right\}\\
&=& \frac{1}{2\pi}\frac{\sqrt{x(4-x)}}{x}.\end{eqnarray*}
\end{example}

\bigskip
\hrule
\bigskip
\noindent 2010 {\it Mathematics Subject Classification}:
Primary 42C05; Secondary
11B83, 11C20, 15B05, 15B36, 33C45.

\noindent \emph{Keywords:} Legendre polynomials, Hermite polynomials, integer sequence, orthogonal polynomials, moments,
Riordan array, Hankel determinant, Hankel transform.

\bigskip
\hrule
\bigskip
\noindent Concerns sequences
\seqnum{A000007},
\seqnum{A000045},
\seqnum{A000108},
\seqnum{A000262},
\seqnum{A001405},
\seqnum{A007318},
\seqnum{A009766},
\seqnum{A021009},
\seqnum{A033184},
\seqnum{A053121},
\seqnum{A094587},
\seqnum{A094816},
\seqnum{A111596},
\seqnum{A111884},
\seqnum{A119467},
\seqnum{A119879},
\seqnum{A155585}.

\end{document}